\documentclass[a4paper,12pt,reqno]{amsart}
\usepackage[utf8]{inputenc}
\usepackage[T1]{fontenc}
\usepackage{lmodern}
\usepackage[english]{babel}
\usepackage{amsmath,a4wide}
\usepackage{xfrac}
\usepackage{stmaryrd,mathrsfs,bm,amsthm,mathtools,yfonts,amssymb,color,braket,booktabs,graphicx,graphics,amsfonts}
\usepackage{latexsym,microtype,indentfirst,hyperref}
\usepackage{xcolor}
\usepackage{courier}
\usepackage[colorinlistoftodos]{todonotes}

\begingroup
\newtheorem{theorem}{Theorem}[section]
\newtheorem{lemma}[theorem]{Lemma}
\newtheorem{proposition}[theorem]{Proposition}
\newtheorem{corollary}[theorem]{Corollary}
\endgroup

\theoremstyle{definition}

\newtheorem{remark}[theorem]{Remark}

\newtheorem{problema}[theorem]{Question}

\newtheorem{taggedtheoremx}{Statement}
\newenvironment{taggedtheorem}[1]
 {\renewcommand\thetaggedtheoremx{#1}\taggedtheoremx}
 {\endtaggedtheoremx}

\numberwithin{equation}{section}
\numberwithin{subsection}{section}

\newcommand{\mres}{\mathbin{\vrule height 1.6ex depth 0pt width
0.13ex\vrule height 0.13ex depth 0pt width 1.3ex}}

\newcommand{\Z}{\mathbb{Z}}

\newcommand{\R}{\mathbb{R}}

\newcommand{\Ha}{\mathcal{H}}
\newcommand{\spt}{\mathrm{spt}}
\newcommand{\Lip}{\mathrm{Lip}}

\newcommand{\D}{\mathscr{D}}
\newcommand{\Rc}{\mathscr{R}}
\newcommand{\I}{\mathscr{I}}
\newcommand{\F}{\mathscr{F}}

\newcommand{\M}{\mathbf{M}}
\newcommand{\Fl}{\mathbf{F}}

\newcommand{\modp}{{\rm mod}(p)}
\newcommand{\moddue}{{\rm mod}(2)}

\title{On the structure of flat chains modulo $p$}

\author{Andrea Marchese \and Salvatore Stuvard}

\newcommand{\Addresses}{{% additional braces for segregating \footnotesize
  \bigskip
  \footnotesize

  A. M. \and S. S., \textsc{Universit\"at Z\"urich, Winterthurerstrasse 190, CH-8057 Z\"urich}
  \par\nopagebreak
  
  \textit{E-mail address}, A. M.: \texttt{andrea.marchese@math.uzh.ch}

  \medskip

  \textit{E-mail address}, S. S.: \texttt{salvatore.stuvard@math.uzh.ch}
   
}}

\begin{document}

\begin{abstract}
In this paper, we prove that every equivalence class in the quotient group of integral $1$-currents modulo $p$ in Euclidean space contains an integral current, with quantitative estimates on its mass and the mass of its boundary. Moreover, we show that the validity of this statement for $m$-dimensional integral currents modulo $p$ implies that the family of $(m-1)$-dimensional flat chains of the form $pT$, with $T$ a flat chain, is closed with respect to the flat norm. In particular, we deduce that such closedness property holds for $0$-dimensional flat chains, and, using a proposition from \textit{The structure of minimizing hypersurfaces mod $4$} by Brian White, also for flat chains of codimension $1$.

\vspace{4pt}

\vspace{4pt}
\end{abstract}

\maketitle

\section{Introduction} \label{S1}

The theory of \emph{integral currents} was born in the 1960's after the work of Federer and Fleming \cite{FF60} out of the desire to solve Plateau's problem. Integral currents were thus introduced to provide a mathematical framework where the existence of orientable surfaces minimizing the volume among those spanning a given contour could be rigorously proved by direct methods in any dimension and codimension. 

In order to deal with non-orientable surfaces, Ziemer \cite{Zie62} introduced the notion of \emph{integral currents modulo $2$}. Further generalizations, such as \emph{integral currents modulo $p$} and \emph{flat chains modulo $p$}, were considered in order to treat a wider class of surfaces which can be realized, for instance, as soap films. An interesting property of such surfaces is that they can develop singularities in low codimension, unlike the classical solutions to Plateau's problem (see, for instance, \cite{Mor86} and \cite{WH86}). 

Moreover, integral currents, flat chains and their generalizations have proved to be flexible enough to describe and tackle similar problems in more abstract settings (see, in particular, \cite{WH99}, \cite{AK00}, \cite{DPH12} and \cite{AG13}).

Despite the substantial interest in the subject, the very structure of flat chains and integral currents modulo $p$ is yet to be completely understood. The initial idea is to define flat chains modulo $p$ by identifying currents which differ by $pT$, where $T$ is a ``classical'' flat chain. This definition, however, has one major drawback: the closedness of the classes with respect to the flat norm is a-priori not guaranteed. Hence, it is more convenient to define the classes of flat chains modulo $p$ as the flat closure of the equivalence classes mentioned above. The equivalence of the two definitions is still an open problem.

A second issue regards the structure of integral currents modulo $p$. They are defined as flat chains modulo $p$ with finite $p$-mass and finite $p$-mass of the boundary. It is not known whether each equivalence class contains at least one classical integral current.

In this work, we specifically address these two problems. In Section \ref{S2} we recall the basic terminology and the main results about classical flat chains and integral currents. Flat chains and integral currents modulo $p$ are introduced in Section \ref{S3}, where we also formulate two questions related to the two above problems and collect some partial answers from the literature. Finally, in Section \ref{S4} we throw light on the connection between the two questions, and we provide a positive answer to the second one in the case of $1$-dimensional currents. Moreover, we give an example illustrating how it is possible to produce situations in which the answer to the second question is negative in higher dimension.

\subsection*{Acknowledgements.} The authors would like to thank Camillo De Lellis for having posed the problem and for helpful discussions. A. M. and S. S. are supported by the ERC-grant ``Regularity of area-minimizing currents'' (306247).

\section{Classical results on flat chains} \label{S2}

In what follows we recall the basic terminology related to the theory of currents. We refer the reader to the introductory presentation given in \cite{Mor09} or to the standard textbooks \cite{Sim83}, \cite{KP08} for further details. The most complete reference remains the treatise \cite{Fed69}. 

\subsection{Currents.}
An \emph{$m$-dimensional current} $T$ in $\R^n$ ($m \leq n$) is a continuous linear functional on the space $\D^{m}(\R^n)$ of smooth compactly supported differential $m$-forms in $\R^n$, endowed with a locally convex topology built in analogy with the topology on $C^{\infty}_{c}(\R^n)$ with respect to which distributions are dual. 

The \emph{boundary} of $T$ is the $(m-1)$-dimensional current $\partial T$ defined by
\[
\langle \partial T, \omega \rangle := \langle T, d\omega \rangle
\]
for any smooth compactly supported $(m-1)$-form $\omega$. The \emph{mass} of $T$, denoted by $\M(T)$, is the (possibly infinite) supremum of $\langle T, \omega \rangle$ over all forms $\omega$ with $|\omega| \leq 1$ everywhere. 

The \emph{support} of a current $T$, denoted $\spt(T)$, is the intersection of all closed sets $C$ in $\R^n$ such that $\langle T, \omega \rangle = 0$ whenever $\omega\equiv 0$ on $C$.

\subsection{Rectifiable currents.}
A subset $E \subset \R^n$ is said to be $m$-\emph{rectifiable} if $\Ha^{m}(E) < \infty$ and $E$ can be covered, except for an $\Ha^{m}$-null subset, by countably many $m$-dimensional surfaces of class $C^{1}$. If $E$ is $m$-rectifiable, then a suitable notion of $m$-dimensional \emph{approximate tangent space} to $E$ can be defined for $\Ha^{m}$-a.e. $x \in E$. Such a tangent space will be denoted ${\rm Tan}(E,x)$ and it coincides with the classical tangent space if $E$ is a (piece of a) $C^1$ $m$-surface.   

Let $E$ be an $m$-rectifiable set in $\R^n$. An \emph{orientation} of $E$ is an $m$-vectorfield $\tau$ on $\R^n$ such that $\tau(x)$ is a simple $m$-vector with $|\tau(x)| = 1$ which spans ${\rm Tan}(E,x)$ at $\Ha^{m}$-a.e. point $x$. A \emph{multiplicity} on $E$ is an integer-valued function $\theta$ such that
\[
\int_{E} |\theta| \, d\Ha^{m} < \infty.
\]

For every choice of a triple $(E,\tau,\theta)$ as above, we denote by $T = \llbracket E, \tau, \theta \rrbracket$ the $m$-dimensional current whose action on a form $\omega$ is given by
\[
\langle T, \omega \rangle := \int_{E} \langle \omega(x), \tau(x) \rangle \theta(x) \, d\Ha^{m}(x).
\]
Currents of this type are called \emph{integer rectifiable $m$-currents}. The set of integer rectifiable $m$-currents in $\R^n$ with support in a compact $K \subset \R^n$ will be denoted $\Rc_{m,K}(\R^n)$. The symbol $\Rc_{m}(\R^n)$ will denote the union of $\Rc_{m,K}(\R^n)$ corresponding to all compact subsets $K \subset \R^n$. If $T = \llbracket E, \tau, \theta \rrbracket \in \Rc_{m}(\R^n)$, we denote by $\| T \|$ the measure given by
\[
\|T\|(A) := \int_{A \cap E} |\theta| \, d\Ha^{m} \hspace{0.5cm} \mbox{for every } A \subset \R^n \mbox{ Borel.}
\]
One can check that $\M(T) = \|T\|(\R^n)$ and thus, in particular, integer rectifiable currents have finite mass.\\

With the symbol $\I_{m,K}(\R^n)$ we denote the set of \emph{integral $m$-currents} supported in $K$, that is currents in $\Rc_{m,K}(\R^n)$ with integer rectifiable boundary. The meaning of the symbol $\I_{m}(\R^n)$ is understood. We recall the following fundamental
\begin{theorem}[Boundary Rectifiability, cf. {\cite[Theorem 4.2.16]{Fed69}}] \label{b_rect:thm}
 \begin{equation} \label{b_rect}
\I_{m,K}(\R^n) = \lbrace T \in \Rc_{m,K}(\R^n) \, \colon \, \M(\partial T) < \infty \rbrace.
\end{equation}
\end{theorem}

\subsection{Flat chains.} The set of (integral) \emph{flat $m$-chains} in $\R^n$ with support in a compact $K \subset \R^n$ is denoted by $\F_{m,K}(\R^n)$ and defined by
\begin{equation} \label{chains}
\F_{m,K}(\R^n) := \lbrace T = R + \partial S \, \colon \, R \in \Rc_{m,K}(\R^n), S \in \Rc_{m+1,K}(\R^n) \rbrace.
\end{equation}
We also define the set $\F_{m}(\R^n)$ as the \emph{union} of the sets $\F_{m,K}(\R^n)$ corresponding to all compact subsets $K \subset \R^n$.

For any $T \in \F_{m,K}(\R^n)$, we define the \emph{flat norm}
\begin{equation} \label{flat}
\Fl_{K}(T) := \inf\lbrace \M(R) + \M(S) \, \colon \, R \in \Rc_{m,K}(\R^n), S \in \Rc_{m+1,K}(\R^n) \mbox{ and } T = R + \partial S \rbrace.
\end{equation}
It turns out that $\Fl_{K}$ induces a complete metric $d_{\Fl_K}$ on $\F_{m,K}(\R^n)$ setting
\[
d_{\Fl_K}(T_{1}, T_{2}) := \Fl_{K}(T_{1} - T_{2}).
\] 
Moreover, the mass $\M$ is lower semi-continuous with respect to the convergence in $\F_{m,K}(\R^n)$ induced by $d_{\Fl_{K}}$.

For every $K$, the class $\I_{m,K}(\R^n)$ is dense in $\Rc_{m,K}(\R^n)$ in mass, and consequently $\I_{m,K}(\R^n)$ is dense in $\F_{m,K}(\R^n)$ in flat norm. In fact, the same result holds more in general whenever the ambient space is a closed convex subset $E$ of a Banach space, if working with the general theory of currents in metric spaces introduced by Ambrosio and Kirchheim in \cite{AK00} (cf \cite[Proposition 14.7]{AK11}). In particular, given $T \in \F_{m,K}(\R^n)$ there exist sequences $\{ T_{j} \} \subset \I_{m,K}(\R^n)$, $\{ R_j \} \subset \Rc_{m,K}(\R^n)$, $\{ S_j \} \subset \Rc_{m+1,K}(\R^n)$ such that
\[
T = T_{j} + R_{j} + \partial S_{j} 
\]
and 
\[
\M(R_j) + \M(S_j) \to 0.
\]
If $T$ has finite mass, then the $\partial S_j$'s have finite mass too, and thus $S_j \in \I_{m+1,K}(\R^n)$. Therefore, the currents $T_{j} + \partial S_{j} \in \Rc_{m,K}(\R^n)$ approximate $T$ in mass, and this suffices to conclude that $T \in \Rc_{m,K}(\R^n)$ (cf. \cite[Lemma 27.5]{Sim83}). We have shown the following result:
\begin{theorem}[Rectifiability of flat chains with finite mass, cf. {\cite[Theorem 4.2.16]{Fed69}}] \label{mass_chain:thm}
\begin{equation} \label{mass_chain}
\Rc_{m,K}(\R^n) = \lbrace T \in \F_{m,K}(\R^n) \, \colon \, \M(T) < \infty \rbrace.
\end{equation}
\end{theorem}

We finally recall the following remarkable
\begin{theorem}[Compactness Theorem, cf. {\cite[Theorem 4.2.17]{Fed69}}] \label{compactness:thm}
Let $K \subset \R^n$ be a compact set and $\{T_{j}\}_{j=1}^{\infty} \subset \I_{m,K}(\R^n)$ a sequence of integral currents such that
\begin{equation} \label{compactness:hp}
\sup_{j \geq 1} \{ \M(T_{j}) + \M(\partial T_{j}) \} < \infty.
\end{equation}
Then, there exist $T \in \I_{m,K}(\R^n)$ and a subsequence $\{T_{j_{\ell}}\}$ such that
\begin{equation} \label{compactness:th}
\lim_{\ell \to \infty} \Fl_{K}(T - T_{j_{\ell}}) = 0.
\end{equation}
\end{theorem}

\subsection{Polyhedral chains.} Given an $m$-dimensional simplex $\sigma$ in $\R^n$ with constant unit orientation $\tau$, we denote by $\llbracket \sigma \rrbracket$ the rectifiable current $\llbracket \sigma, \tau, 1 \rrbracket$. Finite linear combinations of (the currents associated with) oriented $m$-simplexes with integer coefficients are called (integral) \emph{polyhedral $m$-chains}. The set of polyhedral $m$-chains in $\R^n$ will be denoted $\mathscr{P}_{m}(\R^n)$. The main motivation for introducing polyhedral chains is the following theorem. 

%We will need the following preliminary definition. Given an $m$-dimensional current $T$ in $\R^n$ and a proper Lipschitz map $f \colon \R^n \to \R^n$, the \emph{push-forward} of $T$ through $f$ is the current $f_{\sharp}T$ defined by
%\[
%\langle f_{\sharp}T, \omega \rangle = \langle T, f^{\sharp}\omega \rangle,
%\] 
%where $f^{\sharp} \omega$ is the \emph{pull-back} of the smooth form $\omega$ through $f$. The following estimate holds (cf. \cite[4.1.14]{Fed69}):
%\begin{equation} \label{pf_mass}
%\M(f_{\sharp}T) \leq \Lip(f)^{k} \M(T).
%\end{equation}

\begin{theorem}[Polyhedral approximation, cf. {\cite[Corollary 4.2.21]{Fed69}}] \label{poly_app}
If $T \in \I_{m}(\R^n)$, ${\varepsilon > 0}$, $K \subset \R^n$ is a compact set such that $\spt(T) \subset {\rm int}K$, then there exists $P \in \mathscr{P}_{m}(\R^n)$, with $\spt(P) \subset K$, such that
\begin{equation} \label{poly_app:eq}
\Fl_{K}(T - P) < \varepsilon, \hspace{0.5cm}
\M(P) \leq \M(T) + \varepsilon, \hspace{0.5cm} \M(\partial P) \leq \M(\partial T) + \varepsilon.
\end{equation}
\end{theorem} 

\section{Flat chains modulo $p$} \label{S3}

In this section, we recall the definitions of the sets of currents with coefficients in $\Z_{p}$, and collect some of the most relevant open questions regarding their structure.

\subsection{Definitions and basic properties.}
Let $p$ be a positive integer. For any ${T \in \F_{m,K}(\R^n)}$, we define
\begin{equation} \label{p_flat}
\begin{split}
\Fl_{K}^{p}(T) := \inf \lbrace \M(R) + \M(S) \, \colon \, &R \in \Rc_{m,K}(\R^n), S \in \Rc_{m+1,K}(\R^n) \mbox{ s.t. } \\
&T = R + \partial S + pQ \mbox{ for some } Q \in \F_{m,K}(\R^n) \rbrace.
\end{split}
\end{equation}

Observe that, since $\I_{m,K}(\R^n)$ is flat-dense in $\F_{m,K}(\R^n)$, the infimum is unchanged if we let $Q$ run in $\I_{m,K}(\R^n)$. Also notice that the inequality $\Fl_{K}^{p}(T) \leq \Fl_{K}(T)$ holds for any $T \in \F_{m,K}(\R^n)$.

Now, we introduce the equivalence relation $\modp$ in $\F_{m,K}(\R^n)$: given $T, \tilde{T} \in \F_{m,K}(\R^n)$, we say that $T = \tilde{T} \, \modp$ in $\F_{m,K}(\R^n)$ if and only if $\Fl_{K}^{p}(T - \tilde{T}) = 0$. The corresponding quotient group will be denoted $\F_{m,K}^{p}(\R^n)$. As in the classical case, $\Fl_{K}^{p}$ induces a distance $d_{\Fl_{K}^{p}}$ which makes $\F_{m,K}^{p}(\R^n)$ a complete metric space. 

It is evident that if $T - \tilde{T} = pQ$ for some $Q \in \F_{m,K}(\R^n)$, then $T = \tilde{T} \, \modp$ in $\F_{m,K}(\R^n)$, but the converse implication is not known (see Question \ref{pb1} below).

We say that two flat $m$-chains $T, \tilde{T} \in \F_{m}(\R^n)$ are equivalent $\modp$ in $\F_{m}(\R^n)$, and we write $T = \tilde{T} \, \modp$ in $\F_{m}(\R^n)$ if there exists a compact $K \subset \R^n$ such that $\Fl_{K}^{p}(T - \tilde{T}) = 0$. The elements of the corresponding quotient group $\F_{m}^{p}(\R^n)$ are called \emph{flat $m$-chains modulo $p$} and they will be denoted by $\left[T \right]$.

\begin{remark}
\begin{itemize}
\item[$(i)$] Note that if $T\in\F_{m}(\R^n)$ and $\spt(T)\subset K$, then it is false in general that $T\in\F_{m,K}(\R^n)$. The simplest counterexample being the $0$-dimensional current obtained as the boundary of (the rectifiable 1-current associated to) a countable union of disjoint intervals $S_i$ contained in $\left[ 0, 1 \right]$ and clustering only at the origin, when $K=\{0\}\cup\bigcup_i \partial S_i$. Nevertheless, it is a consequence of the polyhedral approximation theorem \ref{poly_app} that if $\spt(T) \subset {\rm int} K$, then indeed $T \in \F_{m,K}(\R^n)$ (see also \cite[Theorem 4.2.22]{Fed69}).

\item[$(ii)$] One would expect that the following property holds. If $T=\tilde{T}\, \modp$ in $\F_{m}(\R^n)$, then $T=\tilde{T}\, \modp$ in $\F_{m,K}(\R^n)$, whenever $K$ is a compact set which contains $\spt(T)$ and $\spt(\tilde T)$, and $T,\tilde{T}\in\F_{m,K}(\R^n)$. Nevertheless, the validity of this property does not appear to be obvious for a general compact set $K$.  On the other hand, if $K$ is also convex, the validity of the property is immediate. Indeed, let $K'$ be a compact set such that $T-\tilde{T}=R_j+\partial S_j+pQ_j$ with $R_j \in \Rc_{m,K'}(\R^n), S_j \in \Rc_{m+1,K'}(\R^n)$, $Q_j \in \F_{m,K'}(\R^n)$ and $\M(R_j) + \M(S_j)\leq \frac{1}{j}$. Then, denoting by $\pi$ the (1-Lipschitz) closest-point projection on $K$, and by $\pi_\sharp$ the push-forward operator through $\pi$ (see \cite[Section 4.1.14]{Fed69}), we have that 
$$T-\tilde{T}=\pi_\sharp T-\pi_\sharp\tilde{T}=\pi_\sharp R_j + \partial\pi_\sharp S_j + p\pi_\sharp Q_j,$$
where $\pi_\sharp R_j \in \Rc_{m,K}(\R^n), \pi_\sharp S_j \in \Rc_{m+1,K}(\R^n)$ and $\pi_\sharp Q_j \in \F_{m,K}(\R^n)$. Moreover $\M(\pi_\sharp R_j) + \M(\pi_\sharp S_j)\leq\M(R_j) + \M(S_j)$, hence $T=\tilde{T}\, \modp$ in $\F_{m,K}(\R^n)$.

\item[$(iii)$] Observe that, using the same argument as in $(ii)$, we are able to conclude that if $T\in\F_{m}(\R^n)$ and $\spt(T)\subset K$ then $T\in\F_{m,K}(\R^n)$ when $K$ is convex (or, more in general, whenever there exists a Lipschitz projection onto $K$).
\end{itemize}
\end{remark}

\subsection{Boundary, mass and support modulo $p$.}

It is immediate to see that if ${T = \tilde{T} \, \modp}$ in $\F_{m}(\R^n)$ (resp. in $\F_{m,K}(\R^n)$), then also $\partial T = \partial \tilde{T} \, \modp$ in $\F_{m-1}(\R^n)$ (resp. in $\F_{m-1,K}(\R^n)$), and therefore a boundary operator $\partial$ can be defined also in the quotient groups $\F_{m}^{p}(\R^n)$ (resp. in $\F_{m,K}^{p}(\R^n)$) in such a way that
\begin{equation} \label{p_bound}
\partial \left[ T \right] = \left[ \partial T \right] \hspace{0.5cm} \mbox{for every } T \in \F_{m}(\R^n).
\end{equation}

For $T \in \F_{m}(\R^n)$, we also define its \emph{mass modulo $p$}, or simply $p$-\emph{mass} $\M^{p}(T)$, as the least $t \in \R \cup \{+\infty\}$ such that for every $\varepsilon > 0$ there exists a compact $K \subset \R^n$ and a rectifiable current $R \in \Rc_{m,K}(\R^n)$ satisfting
\begin{equation} \label{p_mass}
\Fl_{K}^{p}(T - R) < \varepsilon \quad \mbox{and} \quad \M(R) \leq t + \varepsilon.
\end{equation}

One has that $\M^{p}(T_1 + T_2) \leq \M^{p}(T_1) + \M^{p}(T_2)$ and $\M^{p}(T) = \M^{p}(\tilde{T})$ if $T = \tilde{T} \, \modp$ in $\F_{m}(\R^n)$. This allows to regard $\M^{p}$ as a functional on the quotient group $\F_{m}^{p}(\R^n)$. Such functional is lower semi-continuous with respect to the $\Fl_{K}^{p}$-convergence for every $K$. 

Finally, we denote by $\spt^{p}(\left[T\right])$ the \emph{support modulo $p$} of $\left[ T \right] \in \F_{m}^{p}(\R^n)$, given by
\begin{equation} \label{p_spt}
\spt^{p}(\left[T\right]) := \bigcap \lbrace \spt(\tilde{T}) \, \colon \, \tilde{T} \in \F_{m}(\R^n), \, \tilde{T} = T \, \modp \mbox{ in } \F_{m}(\R^n) \rbrace.
\end{equation}

\subsection{Rectifiable and integral currents modulo $p$.}

We define now the group $\Rc_{m}^{p}(\R^n)$ of the \emph{integer rectifiable currents modulo $p$} by setting
\begin{equation} \label{p_irc}
\Rc_{m,K}^{p}(\R^n) := \lbrace \left[ T \right] \in \F_{m,K}^{p}(\R^n) \, \colon \, T \in \Rc_{m,K}(\R^n) \rbrace.
\end{equation}
As usual, $\Rc_{m}^{p}(\R^n)$ is the union over $K$ compact of $\Rc_{m,K}^{p}(\R^n)$. Clearly, not all the elements in a class $\left[ T \right] \in \Rc_{m,K}^{p}(\R^n)$ are classical rectifiable currents, but whenever we write $\left[ T \right] \in \Rc_{m,K}^{p}(\R^n)$ we will always implicitly intend that $T$ is a rectifiable representative of its class.\\

A current $R = \llbracket E, \tau, \theta \rrbracket \in \Rc_{m,K}(\R^n)$ is called \emph{representative modulo $p$} if and only if
\[
\| R \|(A) \leq \frac{p}{2} \Ha^{m}(E \cap A) \hspace{0.5cm} \mbox{for every Borel set } A \subset \R^n.
\]
Evidently, this condition is equivalent to ask that 
\[
|\theta(x)| \leq \frac{p}{2} \hspace{0.5cm} \mbox{for $\|R\|$-a.e. } x.
\]
Since obviously for any integer $z$ there exists a (unique) integer $- \frac{p}{2} < \tilde{z} \leq \frac{p}{2}$ with ${z \equiv \tilde{z} \, ({\rm mod}\, p)}$, then for any $T \in \Rc_{m,K}(\R^n)$ there exists an integer rectifiable current $R\in \Rc_{m,K}(\R^n)$ such that $R = T \, \modp$ in $\F_{m,K}(\R^n)$ and $R$ is representative modulo $p$. We immediately conclude that any $T \in \Rc_{m,K}(\R^n)$ can be written as
\begin{equation} \label{decomp}
T = R + p Q,
\end{equation}
where $R, Q \in \Rc_{m,K}(\R^n)$ and $R$ is representative modulo $p$. It is proved in \cite[4.2.26, p. 430]{Fed69} that
\begin{equation} \label{representative}
\M^{p}(T) = \M(R), \hspace{0.5cm} \spt^{p}(\left[T\right]) = \spt(R),
\end{equation}
if $R$ is representative modulo $p$ of the current $T$.

A modulo $p$ version of Theorem \ref{mass_chain:thm} is contained in \cite[(4.2.16)$^{\nu}$, p. 431]{Fed69}: 
\begin{theorem}[Rectifiability of flat chains modulo $p$] \label{p_mass_chain:thm}
\begin{equation} \label{p_mass_chain}
\Rc_{m,K}^{p}(\R^n) = \lbrace \left[ T \right] \in \F_{m,K}^{p}(\R^n) \, \colon \, \M^{p}(\left[T\right]) < \infty \rbrace.
\end{equation}
\end{theorem}
Hence, if $\left[ T \right] \in \F_{m,K}^{p}(\R^n)$ has finite $\M^p$ mass, then there exists $R \in \Rc_{m,K}(\R^n)$ such that $R = T \, \modp$ in $\F_{m,K}(\R^n)$, $\M(R) = \M^{p}(\left[T\right])$, and $\spt(R) = \spt^{p}(\left[T\right]).$\\

Next, we define the group $\I_{m}^{p}(\R^n)$ of the \emph{integral currents modulo $p$} as the union of the groups
\[
\I_{m,K}^{p}(\R^n) := \lbrace \left[ T \right] \in \Rc_{m,K}^{p}(\R^n) \, \colon \, \partial\left[ T \right] \in \Rc_{m-1,K}^{p}(\R^n) \rbrace.
\]

The conclusions about integer rectifiable currents modulo $p$ deriving from the above discussion allow us to say that if $\left[ T \right] \in \I_{m,K}^{p}(\R^n)$ then $\M^{p}(\left[T\right]) < \infty$, $\M^{p}(\partial \left[T\right]) < \infty$ and that there are currents $R \in \Rc_{m,K}(\R^n)$ and $S \in \Rc_{m-1,K}(\R^n)$ such that $R = T \, \modp$ in $\F_{m,K}(\R^n)$ and $S = \partial T \, \modp$ in $\F_{m-1,K}(\R^n)$. In particular, $R$ and $S$ may be chosen to be representative modulo $p$, so that $\M(R) = \M^{p}(T)$ and $\M(S) = \M^{p}(\partial T)$. It is not known whether it is possible to choose $I \in \I_{m,K}(\R^n)$ such that $T = I \, \modp$ in $\F_{m,K}(\R^n)$ (see Question \ref{pb2} below).

A modulo $p$ version of the Boundary Rectifiability Theorem can be straightforwardly deduced from Theorem \ref{p_mass_chain:thm}, as we have:
\begin{theorem}[Boundary Rectifiability modulo $p$, cf. {\cite[(4.2.16)$^{\nu}$]{Fed69}}] \label{p_b_rect:thm}
 \begin{equation} \label{p_b_rect}
\I_{m,K}^{p}(\R^n) = \lbrace \left[T\right] \in \Rc_{m,K}^{p}(\R^n) \, \colon \, \M^{p}(\partial \left[T\right]) < \infty \rbrace.
\end{equation}
\end{theorem} 

We conclude with the following modulo $p$ version of the Polyhedral approximation Theorem \ref{poly_app}, which can be deduced from \cite[(4.2.20)$^{\nu}$]{Fed69}. Since the statement does not appear in \cite{Fed69}, for the reader's convenience we include here the proof.
\begin{theorem}[Polyhedral approximation modulo $p$] \label{p_poly_app}
If $\left[T\right] \in \I_{m}^{p}(\R^n)$, ${\varepsilon > 0}$, $K \subset \R^n$ is a compact set such that $\spt^{p}(\left[T\right]) \subset {\rm int}K$, then there exists $P \in \mathscr{P}_{m}(\R^n)$, with $\spt(P) \subset K$, such that
\begin{equation} \label{p_poly_app:eq}
\Fl_{K}^{p}(T - P) < \varepsilon, \hspace{0.5cm}
\M^{p}(P) \leq \M^{p}(T) + \varepsilon, \hspace{0.5cm} \M^{p}(\partial P) \leq \M^{p}(\partial T) + \varepsilon.
\end{equation}
\end{theorem}

\begin{proof}
Let $T \in \Rc_{m}(\R^{n})$ be a (rectifiable) representative modulo $p$ of $\left[ T \right]$. In particular, by formula \eqref{representative}, $T$ satisfies $\spt(T) = \spt^{p}(\left[ T \right]) \subset {\rm int}K$ and $\M(T) = \M^{p}(T)$. Fix $\varepsilon > 0$, and let $0 < \delta \leq \varepsilon$ be such that $\{ x \in \R^{n} \, \colon \, {\rm dist}(x, \spt(T)) < \delta \} \subset K$. By \cite[Theorem (4.2.20)$^{\nu}$]{Fed69}, there exist $P \in \mathscr{P}_{m}(\R^{n})$ with $\spt(P) \subset K$ and a diffeomorphism $f \in C^{1}(\R^{n}, \R^{n})$ such that:
\begin{itemize}
\item[$(i)$] $\M^{p}(P - f_{\sharp}T) + \M^{p}(\partial P - f_{\sharp}\partial T) \leq \delta$;

\item[$(ii)$] $\Lip(f) \leq 1 + \delta$, and $\Lip(f^{-1}) \leq 1 + \delta$;

\item[$(iii)$] $|f(x) - x| \leq \delta \mbox{ for } x \in \R^{n}$, and $f(x) = x \mbox{ if } {\rm dist}(x, \spt(T)) \geq \delta$.
\end{itemize}

From $(i)$ it readily follows that
\begin{equation} \label{ppa:1}
\M^{p}(P) \leq \delta + \M^{p}(f_{\sharp}T) \leq \delta + (1 + \delta)^{m} \M^{p}(T),
\end{equation}
and analogously
\begin{equation} \label{ppa:2}
\M^{p}(\partial P) \leq \delta + \M^{p}(f_{\sharp}\partial T) \leq \delta + (1 + \delta)^{m-1} \M^{p}(\partial T).
\end{equation}
In order to prove the estimate on the $\Fl^{p}_{K}$ distance, let $h$ be the affine homotopy from the identity map to $f$, i.e. $h(t,x) := (1 - t) x + t f(x)$, and observe that the homotopy formula (see \cite[Section 4.1.9]{Fed69}) yields
\begin{equation} \label{ppa:3}
P - T = P - f_{\sharp}T + \partial \left( h_{\sharp}(\llbracket \left( 0,1 \right) \rrbracket \times T) \right) + h_{\sharp}(\llbracket \left( 0,1 \right) \rrbracket \times \partial T).
\end{equation}
Now, since $\M^{p}(\partial T) < \infty$, there exists a rectifiable current $Z \in \Rc_{m-1,K}(\R^{n})$ such that
\begin{equation} \label{ppa:4}
\Fl_{K}^{p}(\partial T - Z) \leq \delta \hspace{0.2cm} \mbox{ and } \hspace{0.2cm} \M(Z) \leq \M^{p}(\partial T) + \delta.  
\end{equation}
In particular, this implies the existence of $R \in \Rc_{m-1,K}(\R^{n})$, $S \in \Rc_{m,K}(\R^{n})$ and $Q \in \I_{m-1,K}(\R^{n})$ with $\M(R) + \M(S) \leq 2 \delta$ such that $\partial T - Z = R + \partial S + pQ$. If combined with \eqref{ppa:3}, this gives
\begin{equation} \label{ppa:5}
P - T = P - f_{\sharp}T + \partial h_{\sharp}(\llbracket \left( 0,1 \right) \rrbracket \times T) + h_{\sharp}(\llbracket \left( 0,1 \right) \rrbracket \times Z) + h_{\sharp}(\llbracket \left( 0,1 \right) \rrbracket \times (R + \partial S + pQ)).
\end{equation}
Since, again by the homotopy formula,
\[
h_{\sharp}(\llbracket \left( 0,1 \right) \rrbracket \times \partial S) = f_{\sharp}S - S - \partial h_{\sharp}(\llbracket \left( 0,1 \right) \rrbracket \times S),
\]
we can finally re-write equation \eqref{ppa:5} as follows:
\begin{equation} \label{ppa:6}
\begin{split}
P - T = \,& P - f_{\sharp}T \\
&+ h_{\sharp}(\llbracket \left( 0,1 \right) \rrbracket \times (Z + R)) + f_{\sharp}S - S \\
&+ \partial h_{\sharp}(\llbracket \left( 0,1 \right) \rrbracket \times (T - S)) \\
&+ p h_{\sharp}(\llbracket \left( 0,1 \right) \rrbracket \times Q).  
\end{split}
\end{equation}

Therefore, we can finally estimate
\begin{equation} \label{ppa:7}
\begin{split}
\Fl_{K}^{p}(P - T) \leq \,& \Fl_{K}^{p}(P - f_{\sharp}T) \\
&+ \M(h_{\sharp}(\llbracket \left( 0,1 \right) \rrbracket \times (Z + R))) + \M(f_{\sharp}S) + \M(S) \\
&+ \M(h_{\sharp}(\llbracket \left( 0,1 \right) \rrbracket \times (T - S))) \\
\leq & 3\delta + 2\delta (1+\delta)^{m} + \delta (2 + \delta) (\M^{p}(T) + \M^{p}(\partial T) + 3\delta),
\end{split}
\end{equation}
where we have used \cite[Formula (26.23)]{Sim83} to estimate the first and last addenda in the second line. 

The conclusion, formula \eqref{p_poly_app:eq}, clearly follows from \eqref{ppa:1}, \eqref{ppa:2} and \eqref{ppa:7} for a suitable choice of $\delta = \delta(\varepsilon, m, \M^{p}(T), \M^{p}(\partial T))$.
\end{proof}

\subsection{Questions on the structure of flat chains and integral currents modulo $p$.}

As already anticipated, two very natural questions arise about the structure of flat chains and integral currents modulo $p$ (see \cite[4.2.26]{Fed69}).

We fix a compact subset $K \subset \R^n$.

\begin{problema} \label{pb1}
Given $T,\tilde T\in\F_{m,K}(\R^n)$, is it true that $T = \tilde T \, \modp$ in $\F_{m,K}(\R^n)$ if and only if $T-\tilde T= pQ$ for some $Q \in \F_{m,K}(\R^n)$? In other words, using the density of $\I_{m,K}(\R^n)$ in $\F_{m,K}(\R^n)$, the problem is to prove or disprove the following statement. Given three sequences $\{ R_j \} \subset \Rc_{m,K}(\R^n)$, $\{ S_j \} \subset \Rc_{m+1,K}(\R^n)$,  $\{ Q_j \} \subset \I_{m,K}(\R^n)$ such that 
\begin{equation} \label{pb1:hp1}
T- \tilde T= R_{j} + \partial S_{j} + p Q_{j} \hspace{0.5cm} \forall j,
\end{equation}
and
\begin{equation} \label{pb1:hp2}
\lim_{j\to \infty} \left( \M(R_j) + \M(S_j) \right) = 0,
\end{equation}
then $T-\tilde T = p Q$ for some $Q \in \F_{m,K}(\R^n)$.
\end{problema}

\begin{remark} \label{remarkone}
As we shall soon see, the answer to the above question is affirmative if the class $\F_{m,K}(\R^n)$ is replaced by the class $\Rc_{m,K}(\R^n)$: in other words, given integer rectifiable currents $T, \tilde{T}$ one has that $T = \tilde{T} \, \modp$ in $\F_{m,K}(\R^n)$ if and only if $T - \tilde{T} = p Q$ for some $Q \in \Rc_{m,K}(\R^n)$. As a corollary, Question \ref{pb1} admits affirmative answer for $m = n$, since $\F_{n,K}(\R^n) = \Rc_{n,K}(\R^n)$. For $0 \leq m \leq n-1$, the question is widely open. 
\end{remark}

\begin{problema} \label{pb2}
Given $\left[ T \right] \in \I_{m,K}^{p}(\R^n)$, does there exist an \emph{integral} current $I \in \I_{m,K}(\R^n)$ such that $I = T \, \modp$ in $\F_{m,K}(\R^n)$? In other words: is it true that 
\[
\I_{m,K}^{p}(\R^n) = \lbrace \left[ T \right] \, \colon \, T \in \I_{m,K}(\R^n) \rbrace?
\] 
\end{problema}

\begin{remark}
The answer is trivial for $m=0$, since integral and integer rectifiable $0$-dimensional currents are the same class. In \cite[4.2.26, p. 426]{Fed69}, Federer does not really present this issue as a ``question'', but he rather claims that the answer is negative, in general dimension and codimension. Nevertheless, the counterexample he suggests (an infinite sum of disjoint ${\bf RP}^{2}$ in $\R^6$ with the property that the sum of the areas is finite but the sum of the lengths of the bounding projective lines is infinite) is not fully satisfactory (cf. \cite[Problem 3.3]{Bro86}). Indeed, it allows one to negatively answer the question only for very special choices of the set $K$ (in particular, the question remains open when $K$ is a convex set). 
\end{remark}

\subsection{Some partial answers from the literature.}

An immediate consequence of \eqref{representative} is the following: if $T,\tilde T \in \Rc_{m,K}(\R^n)$ are such that $T = \tilde T \, \modp$ in $\F_{m,K}(\R^n)$, then evidently $\M^{p}(T-\tilde T) = \M^{p}(0) = 0$, and hence the representative modulo $p$ of $T-\tilde T$ is $R = 0$ because of \eqref{representative}. Therefore, equation \eqref{decomp} yields $T-\tilde T = p Q$ for some integer rectifiable current $Q\in\Rc_{m,K}(\R^n)$. In conclusion, we have the following
\begin{proposition}
The answer to Question \ref{pb1} is affirmative in the class of integer rectifiable currents. Therefore:
\begin{equation} \label{p_rect}
\Rc_{m,K}(\R^n)\cap\{T\in\F_{m,K}(\R^n)\,:\,T=0\,\modp\text{ in } \F_{m,K}(\R^n)\} = \{pR\,:\,R\in\Rc_{m,K}(\R^n)\}.
\end{equation}
\end{proposition}

In particular, the following corollary holds true:

\begin{corollary}
Let $T, \tilde{T} \in \Rc_{m}(\R^n)$. Then, $T = \tilde{T} \, \modp$ in $\F_{m}(\R^n)$ if and only if $T = \tilde{T} \, \modp$ in $\F_{m,K}(\R^n)$ for every $K$ compact with $\spt(T) \cup \spt(\tilde{T}) \subset K$.
\end{corollary}

\begin{proof}
The ``if'' implication is trivial. For the converse, assume $T = \tilde{T} \, \modp$ in $\F_{m}(\R^n)$ and fix any compact set $K$ such that $\spt(T) \cup \spt(\tilde{T}) \subset K$. By definition, there exists a compact set $K'$ such that $\Fl^{p}_{K'}(T - \tilde{T}) = 0$, which, by the above proposition, implies
\[
T - \tilde{T} = pQ
\]
for some $Q \in \Rc_{m,K'}(\R^n)$. Note that since $T - \tilde{T}$ is supported in $K$, so is $Q$, and thus $\Fl_{K}(T - \tilde{T}) = 0$, i.e. $T = \tilde{T} \, \modp$ in $\F_{m,K}(\R^n)$.
\end{proof}

From now on, by virtue of the previous corollary, for rectifiable currents $T$ and $\tilde{T}$ in $\F_{m}(\R^n)$ which are equivalent modulo $p$ we will just write $T = \tilde{T} \, \modp$ without specifying in which class the equivalence relation is meant.
\\

In codimension $0$, B. White \cite{WH79} gave an affirmative answer to Question \ref{pb2}.
\begin{theorem}[cf. {\cite[Proposition 2.3]{WH79}}] \label{White}
Let $T\in \Rc_{n,K}(\R^n)$. Then, $\left[ T \right] \in \I_{n,K}^{p}(\R^n)$ if and only if the select representative modulo $p$ of $T$ is an integral current.
\end{theorem}

The \emph{select} representative modulo $p$ of a rectifiable current $T=\llbracket E, \tau, \theta \rrbracket$ is the unique $T'=\llbracket E, \tau, \theta' \rrbracket$ representative modulo $p$ of $T$ with multiplicity $\theta' \in \left( - \frac{p}{2}, \frac{p}{2} \right]$.

White's proof relies on the following:
\begin{proposition} \label{S_n}
If $T \in \Rc_{n,K}(\R^n)$ is a select representative modulo $p$, then
\begin{equation} \label{claim}
\M(\partial T) \leq (p - 1) \M^{p}(\partial T).
\end{equation}
\end{proposition}

We sketch the proof of Theorem \ref{White}, having shown Proposition \ref{S_n}. Take $\left[ T \right] \in \I_{n,K}^{p}(\R^n)$, and let $T'$ be the unique select representative modulo $p$ of $T$. A priori, $T'$ is just an integer rectifiable current. On the other hand, since $\left[ T \right]$ is integral, $\M^{p}(\partial T)$ is finite by \eqref{p_mass_chain}. Then Proposition \ref{S_n} implies that $\M(\partial T')$ is finite. Hence, $T'$ is integral because of \eqref{b_rect}.\\

Unfortunately, in order to carry on the argument that White uses to prove Proposition \ref{S_n}, the codimension $0$ assumption is indispensable. The idea is the following. Firstly, Theorem \ref{p_poly_app} allows one to reduce the problem to the case of polyhedral chains. Now, for any given polyhedral chain $T$ which is a select representative modulo $p$ one writes $T = \llbracket \R^n, \mathbf{e_n}, \theta \rrbracket$, where $\mathbf{e_n}$ is the constant standard orientation of $\R^n$ and $\theta$ is a summable, piecewise constant, integer-valued function with values in $\left( - \frac{p}{2}, \frac{p}{2} \right]$. Then, White makes the following key observation: since the codimension is $0$, if $Z$ is a polyhedron in $\partial T$ then for $\Ha^{n-1}$-a.e. $x\in Z$ the multiplicity at $x$ is the difference of the values that the function $\theta$ takes on the two sides of $Z$ (with the correct sign), whose absolute value is in fact bounded by $p-1$ (because $T$ is a select representative modulo $p$).\\

In the next section, we will show that the validity of a statement like the one in Proposition \ref{S_n} is in fact the key not only for giving an affirmative answer to Question \ref{pb2}, but also for positively answering Question \ref{pb1}. Furthermore, we will answer Question \ref{pb2} in dimension $m=1$.

%\textcolor{blue}{
%A third interesting problem is in \cite[4.2.26, p. 432]{Fed69}.
%\begin{problema} \label{pb3}
%By the above discussion, we have seen that if $\left[ T \right] \in \Rc_{k}^{p}(\R^n)$ then there exists an integer rectifiable current $R \in \Rc_{k}(\R^n)$ such that $R = T \, \modp$ and $\| T \|_{p} = \| R \| \leq (p/2) \Ha^{k}$. It is not known if there exist currents $T \in \F_{k}(\R^n)$ such that $\M_{p}(T) = + \infty$ but $\Ha^{k}(\spt_{p}(T)) = 0$.
%\end{problema}
%}

\section{Main results} \label{S4}

In this section, we will further analyze Questions \ref{pb1} and \ref{pb2}. First, we point out that the two questions are, in fact, connected. 

\subsection{Connection between Questions \ref{pb1} and \ref{pb2}.}

For every $K \subset \R^n$ compact, consider the following family of statements $\mathcal{S}_{m}$, for $m = 1,\dots, n$.

\begin{taggedtheorem}{$\mathcal{S}_{m}$} \label{S_k}
There exists a constant $C = C(m,n,p,K)$ with the following property. For any $\left[ S \right] \in \Rc_{m,K}^{p}(\R^n)$ there exists a current $\tilde{S} \in \Rc_{m,K}(\R^n)$ with $\tilde{S} = S \, \modp$ and such that
\[
\M(\partial \tilde{S}) \leq C \M^{p}(\partial S).
\]
\end{taggedtheorem}

Using Theorem \ref{p_poly_app}, it is easy to see that the validity of Statement \ref{S_k} follows from the validity of a slightly stronger property for polyhedral chains, which, on the other hand, might be easier to check.

\begin{taggedtheorem}{$\mathcal{P}_{m}$} \label{S_k_poly}
There exists a constant $C = C(m,n,p)$ \emph{independent of $K$} with the following property. For any $P \in \mathscr{P}_{m}(\R^n)$ with $\spt(P) \subset K$, there exists a current $\tilde{P} \in \mathscr{P}_{m}(\R^n)$, with $\tilde{P} = P \, \modp$ and $\spt(\tilde P) \subset K$ such that 
\[
\M(\partial \tilde{P}) \leq C \M^{p}(\partial P);\hspace{0.5cm} \M(\tilde{P}) \leq C \M^{p}(P).
\]
\end{taggedtheorem}

\begin{proposition}\label{equiv_statements}
The validity of Statement $\mathcal{P}_{m}$ implies that of Statement $\mathcal{S}_{m}$.
\end{proposition}
\begin{proof}
Let $\left[ S \right] \in \Rc_{m,K}^{p}(\R^n)$. We can assume that $\M^p(\left[ \partial S \right])$ is finite, otherwise the conclusion of Statement $\mathcal{S}_{m}$ is trivial. By Theorem \ref{p_poly_app}, for every $j=1,2,\ldots$ there exists $P_j\in \mathscr{P}_{m}(\R^n)$ such that, denoting 
$$K_j:=\left\lbrace x\in\R^n:{\rm{dist}}(x,K)\leq\frac{1}{j}\right\rbrace,$$ 
one has $\spt(P_{j}) \subset K_{j}$ and
\begin{equation} \label{e1}
\Fl_{K_j}^{p}(S - P_j) < \frac{1}{j}, \hspace{0.5cm}
\M^{p}(P_j) \leq \M^{p}(S) + \frac{1}{j}, \hspace{0.5cm} \M^{p}(\partial P_j) \leq \M^{p}(\partial S) + \frac{1}{j}.
\end{equation}
Now, by Statement $\mathcal{P}_{m}$ there exist a constant $C$ (which does not depend on $j$) and a sequence $\{\tilde P_j\}$ of polyhedral chains with $\tilde P_j = P_j \, \modp$ and $\spt(\tilde P_j)\subset K_j$ such that 
\[
\M(\partial \tilde P_j) \leq C \M^{p}(\partial P_j);\hspace{0.5cm} \M(\tilde{P_j}) \leq C \M^{p}(P_j).
\]
Combining this with \eqref{e1}, we get
\[ 
\sup_{j \geq 1} \{ \M(\tilde P_j) + \M(\partial \tilde P_j) \} \leq C(\M^p(S)+\M^p(\partial S)+2)< \infty.
\]
Then, by the Compactness Theorem \ref{compactness:thm} there exist $\tilde S \in \I_{m,K_1}(\R^n)$ and a subsequence $\{\tilde P_{j_h}\}$ such that
\begin{equation}\label{e2}
\lim_{h \to \infty} \Fl_{K_1}(\tilde S - \tilde P_{j_h}) = 0.
\end{equation}
Moreover by the lower semi-continuity of the mass, it holds
$$\M(\partial \tilde{S}) \leq C \M^{p}(\partial S);\hspace{0.5cm} \M(\tilde{S}) \leq C \M^{p}(S)$$
and we claim that $\spt (\tilde S)\subset K$. Indeed, take $x\in\R^n\setminus K$. We will prove that there exists a closed set $C$ such that $x\not\in C$ and $\langle \tilde S,\omega\rangle=0$ whenever $\omega\equiv 0$ on $C$, which implies that $x\not\in\spt(\tilde S)$. Fix $\ell$ such that $x\not\in K_{j_\ell}$ and let $C:=K_{j_\ell}$. Let $\omega$ be an $m$-form with $\omega\equiv 0$ on $C$. Since for every $h\geq\ell$ it holds $\spt(\tilde P_{j_h})\subset C$, we have 
\begin{equation}\label{fava}
\langle \tilde P_{j_h},\omega\rangle=0,\quad \text{for every $h\geq\ell$}.
\end{equation}
On the other hand, by \eqref{e2}, for every $\varepsilon>0$ there exists $h\geq\ell$ such that we can write $\tilde S-\tilde P_{j_h}=R+\partial Q$ for some $R\in\mathscr{R}_{m,K_1}(\R^n)$ and $Q\in\mathscr{R}_{m+1,K_1}(\R^n)$ with $\M(R)+\M(Q)\leq\varepsilon$. Hence it holds
$$\langle \tilde S-\tilde P_{j_h},\omega\rangle = \langle R,\omega\rangle+\langle \partial Q,\omega\rangle\leq\M(R)\|\omega\|_{\infty}+\M(Q)\|d\omega\|_{\infty}\leq\varepsilon(\|\omega\|_{\infty}+\|d\omega\|_{\infty}).$$
Hence by \eqref{fava} $\langle \tilde S,\omega\rangle=0$, which completes the proof of the claim.\\

Finally, we show that $\tilde S=S\,\modp$. To this aim, for every $h=1,2,\ldots$, we compute
$$\Fl^p_{K_1}(\tilde S-S) \leq \Fl^p_{K_1}(\tilde S - \tilde P_{j_h})+\Fl^p_{K_1}(\tilde P_{j_h}-S)\leq \Fl_{K_1}(\tilde S - \tilde P_{j_h})+\Fl^p_{K_{{j_h}}}(\tilde P_{j_h}-S),$$
which by \eqref{e1} and \eqref{e2} tends to $0$ when $h$ tends to $\infty$.
\end{proof}

\begin{remark}
It follows from the above proof that if the Statement $\mathcal{P}_{m}$ holds true then the Statement $\mathcal{S}_{m}$ holds true with the same constant $C$. In particular, the constant would not depend on the compact set $K$.
\end{remark}

Clearly, if the Statement $\mathcal{S}_{m}$ is true then every $m$-dimensional integral current modulo $p$ in $K$ has an integral representative in $K$, and thus the answer to Question \ref{pb2} is affirmative in dimension $m$. The next theorem shows that, in fact, the validity of $\mathcal{S}_{m}$ has important consequences on Question \ref{pb1} as well.

\begin{theorem} \label{link:thm}
If $\mathcal{S}_{m}$ holds true, then Question \ref{pb1} has affirmative answer in $\F_{m-1,K}(\R^{n})$.
\end{theorem}

\begin{proof}
It is sufficient to prove that if $T \in \F_{m-1,K}(\R^n)$ is a flat $(m-1)$-chain such that $T = 0 \, \modp$ in $\F_{m-1,K}(\R^n)$, then $T=pQ$ for some $Q\in\F_{m-1,K}(\R^n)$. Let ${\{ R_{j} \} \subset \Rc_{m-1,K}(\R^n)}$, $\{ S_{j} \} \subset \Rc_{m,K}(\R^n)$ and $\{ Q_{j} \} \subset \I_{m-1,K}(\R^n)$ be such that
\begin{equation} \label{link:1}
T = R_{j} + \partial S_{j} + p Q_{j} \hspace{0.5cm} \forall j 
\end{equation}
and 
\begin{equation} \label{link:2}
\lim_{j \to \infty} \left( \M(R_j) + \M(S_j) \right) = 0.
\end{equation}
Conditions \eqref{link:1} and \eqref{link:2} are equivalent to say that the currents $p Q_{j}$ converge to $T$ in flat norm $\Fl_{K}$. We want to conclude from this that $T = p Q$ for some $Q \in \F_{m-1,K}(\R^n)$. In other words, we are looking for a result of closedness of the currents of the form $p Q$ with respect to flat convergence. Now, observe the following. For every $j$, the current $R_{j}$ is rectifiable. Therefore, we can write 
\begin{equation} \label{link:3}
R_{j} = \tilde{R}_{j} + p V_{j},
\end{equation}
with $V_j \in \Rc_{m-1,K}(\R^n)$ and $\tilde{R}_j$ representative modulo $p$. In particular, this implies that
\begin{equation} \label{link:4}
\M(\tilde{R}_{j}) = \M^{p}(R_j) \leq \M(R_j) \to 0.
\end{equation}
Also the currents $S_{j}$ are rectifiable, and of dimension $m$. Since $\mathcal{S}_{m}$ holds true, for every $j$ we can let $\tilde{S}_{j}$ be the representative of $\left[ S_{j} \right]$ given in there, so that
\begin{equation} \label{link:4bis}
S_{j} = \tilde{S}_{j} + p Z_{j}
\end{equation}
with $\tilde{S}_{j}, Z_{j} \in \Rc_{m,K}(\R^n)$, and
\begin{equation} \label{link:5}
\M(\partial \tilde{S}_{j}) \leq C(m,n,p,K) \M^{p}(\partial S_{j}).
\end{equation}
Now, since $\M^{p}(T - p Q_j) = \M^{p}(T) = 0$ for every $j$ and $\M^{p}(R_j) \to 0$, we deduce from \eqref{link:1} that also $\M^{p}(\partial S_j) \to 0$, and therefore also $\M(\partial \tilde{S}_{j}) \to 0$.

Thus, the above argument produces the following: modulo replacing $Q_j \in \I_{m-1,K}(\R^n)$ with $\tilde{Q}_{j} := Q_{j} + V_{j} + \partial Z_{j} \in \F_{m-1,K}(\R^n)$, we can replace \eqref{link:1} with
\begin{equation} \label{link:6}
T = \tilde{R}_j + \partial \tilde{S}_j + p \tilde{Q}_j \hspace{0.5cm} \forall j,
\end{equation}
and \eqref{link:2} with the stronger
\begin{equation} \label{link:7}
\lim_{j\to\infty} \left( \M(\tilde{R}_j) + \M(\partial \tilde{S}_j) \right) = 0,
\end{equation}
that is the currents $p \tilde{Q}_{j}$ are now approximating $T$ \emph{in mass}.

The problem, now, reduces to proving that the subset of flat chains in $\F_{m-1,K}(\R^n)$ of the form $pQ$ is closed with respect to convergence in mass: this question, though, is evidently much easier than the previous one, and it turns out to always have affirmative answer. Indeed, let $\{ Q_{j} \}_{j=1}^{\infty} \subset \F_{m-1,K}(\R^n)$ be a sequence of flat chains such that ${\M(T - pQ_{j}) \to 0}$. In particular, this would imply that the sequence $\{pQ_j\}$ is a Cauchy sequence in mass. Therefore, the sequence $\{Q_{j}\}$ is also a Cauchy sequence in mass, and in fact also in the flat norm $\Fl_{K}$, since $\Fl_{K}(T) \leq \M(T)$ for any $T \in \F_{m-1,K}(\R^n)$.\footnote{If $\M(T) = \infty$ there is of course nothing to prove. On the other hand, if $\M(T) < \infty$ then $T$ is integer rectifiable, and hence it is a competitor for the decomposition in the definition of the flat norm.} So, by completeness there is $Q \in \F_{m-1,K}(\R^n)$ such that $\Fl_{K}(Q - Q_j) \to 0$. This also implies $\Fl_{K}(pQ - pQ_j) \to 0$, since $\Fl_{K}(nT) \leq n \Fl_{K}(T)$ in general. So, $pQ$ is a flat limit of the sequence $pQ_j$. By uniqueness of the limit, one therefore has to conclude $T = pQ$.
\end{proof}

\begin{corollary} \label{cod1_cor}
The answer to Question \ref{pb1} is positive for $m = n-1$.
\end{corollary}

\begin{proof}
It immediately follows from Theorem \ref{link:thm}, since $\mathcal{S}_{n}$ is Proposition \ref{S_n}.
\end{proof}

\subsection{Answer to Question \ref{pb2} in dimension $m=1$.}

\begin{theorem}\label{t:main}
The answer to Question \ref{pb2} is positive for $m=1$.
\end{theorem}
In the proof, we will use the following elementary fact.
\begin{lemma}\label{l:chain} Let $P \in \mathscr{P}_{1}(\R^n)$ have positive multiplicities. Let $z$ be a point in $\spt(\partial P)$. Then one can select a finite sequence of oriented segments $S_1,\ldots, S_N$ supported in the support of  $P$ such that: 
\begin{enumerate}
\item the orientation of each segment $S_i$ coincides with the orientation of $P$ on $S_i$;
\item the second extreme of $S_{i}$ coincides with the first extreme of $S_{i+1}$, for $i=1,\ldots, N-1$;
\item If the multiplicity of $\partial P$ at $z$ is negative, then the first extreme of $S_1$ is $z$ and the second extreme of $S_N$ is a point $x$ of the support of $\partial P$ with positive multiplicity. Vice versa, if the multiplicity of $\partial P$ at $z$ is positive, then the first extreme of $S_1$ is a point $x$ of the support of $\partial P$ with negative multiplicity and the second extreme of $S_N$ is $z$;
\item $S_i\neq S_j$ for $i\neq j$.
\end{enumerate}
\end{lemma}
\begin{proof}
Assume without loss of generality that the multiplicity of $\partial P$ at $z$ is negative. Since the multiplicities on $P$ are all positive, then among the (finitely many) segments defining the support of $P$ there is at least a segment $S_1$ whose first extreme is $z$ such that 
\begin{equation}\label{e:spezza_massa}
\M(P)=\M(P - \llbracket S_1 \rrbracket)+\M(\llbracket S_1 \rrbracket),
\end{equation}
If the second extreme $y$ of $S_1$ is not a point with positive multiplicity of $\partial P$, it is a point of negative multiplicity of $\partial (P-\llbracket S_1 \rrbracket)$, hence the procedure can be repeated with $P-\llbracket S_1 \rrbracket$ in place of $P$ and $y$ in place of $z$. The procedure has to terminate in a finite number of steps, because of \eqref{e:spezza_massa} and the fact that the mass of each $\llbracket S_{i} \rrbracket$ is bounded from below. When the procedure ends, one can easily see that the ordered sequence of segments collected satisfies properties $(1)-(3)$. Property $(4)$ is not necessarily satisfied. If a certain segment $S'$ is repeated in the procedure, it is sufficient to eliminate from the sequence one copy of $S'$ and all the segments appearing between two repetitions of $S'$. After this elimination, the sequence satisfies also property $(4)$.  
\end{proof}

\begin{proof}[Proof of Theorem \ref{t:main}] By Proposition \ref{equiv_statements} it is sufficient to prove Statement $\mathcal{P}_1$. Consider $P\in \mathscr{P}_1(\R^n)$.
Firstly we choose a representative $Q\in \mathscr{P}_1(\R^n)$ modulo $p$ of $P$ with multiplicities in $\{1,\ldots,p-1\}$. Clearly we have $\M(Q)\leq (p-1)\M^p(P)$, but at the moment we have no control on $\M(\partial Q)$. Hence, we want to replace $Q$ with another representative $\tilde P\in \mathscr{P}_1(\R^n)$ of $P$, for which we can control both the mass and the mass of the boundary. More precisely, we want to find a representative $\tilde P$ with multiplicities in $\{1,\ldots,p-1\}$ and with the multiplicities of $\partial \tilde P$ in $\{-(p-1),\ldots,p-1\}$.\\

Consider a point $z \in {\rm{spt}}(\partial Q)$ with multiplicity $\theta_{z}$ such that $|\theta_{z}|\geq p$. Without loss of generality, we can assume $\theta_{z} < 0$. Given that the multiplicities on $Q$ are all positive, we can use Lemma \ref{l:chain} to select a finite sequence of oriented segments $S_1,\ldots, S_N$ supported in the support of  $Q$, satisfying properties $(1)-(4)$ (with $Q$ in place of $P$).

Once we have found such a sequence of segments, denote by $Q^1$ the polyhedral current obtained from $Q$ by changing on every segment $S_i$ both the orientation and the multiplicity from $\theta_i$ to $\theta_i^1:=(p-\theta_i)$.
Clearly $Q^1$ has still multiplicities in $\{1,\ldots,p-1\}$. Moreover, if $\theta_z^1$ denotes the multiplicity of $ \partial Q^1$ at $z$ then one has $|\theta^1_z|=|\theta_z|-p$. On the other hand, if $x$ denotes the other endpoint of the chain of segments as in $(3)$ of Lemma \ref{l:chain} and $\theta_x$, $\theta_x^1$ are the multiplicities of $\partial Q$ and $\partial Q^1$ at $x$ respectively, then it holds $\theta_{x}^{1} = \theta_{x} - p$. Now, since by Lemma \ref{l:chain}(3) it holds $\theta_{x} \geq 1$, it follows that $\theta_{x}^{1} = (\theta_{x} - p)  \in \left[ 1 - p, \theta_{x} \right]$. Hence, $|\theta_x^1|\leq|\theta_x|+p-2$.

Therefore, one has
\begin{equation} \label{cade_massa}
\M(\partial Q^1)\leq \M(\partial Q)-2.
\end{equation}
If possible, we repeat the procedure above with $Q^{1}$ in place of $Q$, producing a new polyhedral current $Q^{2}$. By formula \eqref{cade_massa}, the procedure can be iterated only a finite number $M$ of times. The corresponding $\tilde P:=Q^M$ has the required property, because any point $z \in \spt(\partial Q^M)$ has multiplicity $|\theta_{z}| \leq p-1$. Obviously we have 
$$\M(\tilde P)\leq (p-1)\M^p(P)\quad{\rm{and}}\quad \M(\partial \tilde P)\leq (p-1)\M^p(\partial P),$$
and the proof is complete.
\end{proof}

Since we have actually proved the Statement $\mathcal{P}_{1}$, it follows from Proposition \ref{equiv_statements} that the Statement $\mathcal{S}_{1}$ holds true. By virtue of Theorem \ref{link:thm}, we can therefore deduce the following
\begin{corollary} \label{dim0_cor}
The answer to Question \ref{pb1} is positive for $m = 0$.
\end{corollary}

\subsection{Negative answer to Question \ref{pb2} in general dimension.}

It is evident that the choice of the compact set $K$ could be crucial for establishing an answer to Question \ref{pb2}. In the spirit of the counterexample suggested by Federer in \cite{Fed69}[4.2.26, p. 426] (see Remark \ref{remarkone} above), we provide a negative answer to the question, proving the existence of a compact subset $K \subset \R^5$ and a current $\left[ T \right] \in \I_{2,K}^{2}(\R^5)$ with $\partial T = 0 \, \moddue$ such that there exists no $I \in \I_{2,K}(\R^5)$ with $I = T \, \moddue$. Nevertheless, for a different choice of a compact $K' \supset K$ we can exhibit an integral current $I' \in \I_{2,K'}(\R^5)$ with $\partial I' = 0$ and $I' = T \, \moddue$.

In what follows, we will let $\mathcal{K}$ be the embedded Klein bottle in $\R^{4}$ (in particular, $\mathcal{K}$ is a non-orientable compact two dimensional surface without boundary in $\R^4$). There exist a closed curve $\gamma$ and an integral current $S := \llbracket \mathcal{K}, \tau, 1 \rrbracket \in \I_{2,\mathcal{K}}(\R^{4})$ such that the set of discontinuity points of $\tau$ coincides with $\gamma$. In particular, $\partial S$ is the integral current $\llbracket \gamma, \tau_{\gamma}, 2 \rrbracket$, $\tau_{\gamma}$ being the orientation of $\gamma$ naturally induced by $\tau$. We let $\left[ S \right] \in \I_{2,\mathcal{K}}^{2}(\R^4)$ be the associated current $\moddue$. In particular, $\partial \left[ S \right] = 0$ and $\M^{2}(\left[ S \right]) = \Ha^{2}(\mathcal{K})$.  We have the following, elementary

\begin{lemma} \label{lem_neg}
There exists a constant $c = c(\mathcal{K})$ with the following property. If $R \in \I_{2,\mathcal{K}}(\R^4)$ is such that $R \in \left[ S \right]$, one has
\[
\M(\partial R) \geq c.
\]
\end{lemma}

\begin{proof}
By contradiction, let $\{\alpha_{j}\}_{j=1}^{\infty}$ be a sequence of positive numbers with $\alpha_{j} \searrow 0$, and $\{R_{j}\}_{j=1}^{\infty} \subset \I_{2,\mathcal{K}}(\R^4)$ be such that 
\[
R_{j} \in \left[ S \right] \hspace{0.5 cm} \forall j,
\]
and 
\[
\M({\partial R_{j}}) \leq \alpha_{j}.
\]
We write $R_{j} = \llbracket \mathcal{K}, \tau, \theta_{j} \rrbracket$, and we observe that, since $R_{j} = S \, \moddue$, from \eqref{p_rect} and from the definition of $S$ it follows that
\begin{equation} \label{congruenza}
\theta_{j}(x) \equiv 1 \, ({\rm mod} \, 2) \quad \mbox{for } \Ha^{2}\mbox{-a.e.} \, x \in \mathcal{K}.
\end{equation}

We replace every $R_{j}$ with the integral current $\tilde{R}_{j} = \llbracket \mathcal{K}, \tau, \tilde{\theta}_{j} \rrbracket$, where ${\tilde{\theta}_{j} := {\rm sign}(\theta_{j})}$. Clearly, by \eqref{congruenza} and the definition of $\tilde{\theta}_{j}$, $\tilde{R}_{j} = R_{j} \, \moddue$, and thus $\tilde{R}_{j} \in \left[ S \right]$ for every $j$. Notice, furthermore, that $\M(\tilde{R}_{j}) = \Ha^{2}(\mathcal{K})$ for every $j$, and that 
\begin{equation} \label{red_bdry}
\M(\partial \tilde{R}_{j}) \leq \M(\partial R_{j}) \leq \alpha_{j}.
\end{equation}
In order to show \eqref{red_bdry}, let $U$ be any open set in $\mathcal{K}$ homeomorphic to a two-dimensional disc. Let also $\sigma$ be a fixed continuous orientation on $U$. We have that $R_{j} \mres U = \llbracket U, \sigma, \Theta_{j} \rrbracket$, where $\Theta_{j}$ is the function defined by
\[
\Theta_{j}(x) :=
\begin{cases}
\theta_{j}(x) &\mbox{if } \tau(x) = \sigma(x) \\
- \theta_{j}(x) &\mbox{if } \tau(x) = - \sigma(x).
\end{cases}
\]
As a consequence of \cite[Remark 27.7]{Sim83}, it holds
\[
\M((\partial R_{j}) \mres U) = |D\Theta_{j}|(U),
\]
where $| D \Theta_{j} |$ is the variation of the ${\rm BV}$ function $\Theta_{j}$. Analogously, $\tilde{R}_{j} \mres U = \llbracket U, \sigma, \tilde{\Theta}_{j} \rrbracket$, where $\tilde{\Theta}_{j}$ is the function defined by
\[
\tilde{\Theta}_{j}(x) :=
\begin{cases}
\tilde{\theta}_{j}(x) &\mbox{if } \tau(x) = \sigma(x) \\
- \tilde{\theta}_{j}(x) &\mbox{if } \tau(x) = - \sigma(x).
\end{cases}
\]  
Observe that $\tilde{\Theta}_{j} \equiv {\rm sign}(\Theta_{j})$, and hence
\[
\M((\partial \tilde R_{j}) \mres U) = |D\tilde\Theta_{j}|(U) \leq |D \Theta_{j}|(U) = \M((\partial R_{j}) \mres U),
\]
which completes the proof of \eqref{red_bdry}.

Now, by the Compactness Theorem \ref{compactness:thm} there exists a current $\tilde{R} \in \I_{2,\mathcal{K}}(\R^4)$ and a subsequence (not relabeled) such that
\[
\lim_{j \to \infty} \Fl_{\mathcal{K}}(\tilde{R} - \tilde{R}_{j}) = 0.
\]
Moreover, by the lower semi-continuity of the mass one has $\partial \tilde{R} = 0$. Since the equivalence classes $\moddue$ are closed with respect to the flat convergence, $\tilde{R} \in \left[ S \right]$, which contradicts the fact that $\mathcal{K}$ is not orientable. 
\end{proof}

\begin{remark} \label{rem_neg}
Observe that if $\mathcal{K}_{\lambda}$ is a homothetic copy of $\mathcal{K}$ with homothety ratio $\lambda$, then $c(\mathcal{K}_{\lambda}) = \lambda c(\mathcal{K})$. 
\end{remark}

We finally define the compact set $K \subset \R^{5}$ and the current $\left[ T \right] \in \I_{2,K}^{2}(\R^5)$ as follows. For every $i = 1, 2, \dots$, we let $\Lambda_{i}$ be the homothety on $\R^{4}$ defined by $\Lambda_{i}(x) := \displaystyle \frac{x}{i}$, and $\pi_{i} \colon \R^{4} \to \R^5$ be the isometry $\pi_{i}(x) := \left( \displaystyle \frac{1}{i}, x \right)$. We set
\[
K := \{ 0 \} \cup \bigcup_{i=1}^{\infty} \pi_{i} \circ \Lambda_{i}(\mathcal{K}),
\]
which is evidently compact, and
\[
T := \sum_{i=1}^{\infty} (\pi_{i} \circ \Lambda_{i})_{\sharp}S .
\]

We let $\left[ T \right]$ denote the equivalence class of $T$ modulo $2$. Since $\M^{2}((\pi_{i} \circ \Lambda_{i})_{\sharp}S) = \displaystyle \frac{1}{i^{2}} \Ha^{2}(\mathcal{K})$, then $\left[ T \right]$ is well defined, and in particular $\partial \left[ T \right] = 0$. In the following proposition, we show that the choice of $K$ and $\left[ T \right]$ provides a negative answer to Question \ref{pb2}.

\begin{proposition}
In general, the answer to Question \ref{pb2} is negative.
\end{proposition}

\begin{proof}
Let $K$ and $\left[ T \right]$ be as above, and assume by contradiction that there exists $I \in \I_{2,K}(\R^5)$ with $I \in \left[ T \right]$. Then, the restriction of $I$ to each plane $x_{1} = \displaystyle \frac{1}{i}$ belongs to the class $\left[ (\pi_{i} \circ \Lambda_{i})_{\sharp} S \right]$, and thus by Lemma \ref{lem_neg} and Remark \ref{rem_neg} one has
\[
\M(\partial I) = c(\mathcal{K}) \sum_{i=1}^{\infty} \frac{1}{i} = \infty,
\]
which gives the desired contradiction.
\end{proof}

\begin{remark}
Observe that if we replace $\mathcal{K}$ with $\mathcal{K}' := \mathcal{K} \cup D$, where $D$ is a suitable two-dimensional disc, then Lemma \ref{lem_neg} fails, as there exists $R \in \I_{2,\mathcal{K}'}(\R^4)$ such that $R \in \left[ S \right]$ and $\partial R = 0$. Hence, it is possible to construct an integral representative of $\left[ T \right]$ with support in 
\[
K' := \{0\} \cup \bigcup_{i=1}^{\infty} \pi_{i} \circ \Lambda_{i}(\mathcal{K}').
\] 
\end{remark}

\subsection{Concluding remarks.}

Ambrosio and Wenger proved in \cite[Theorem 4.1]{AW11} a statement similar to our Theorem \ref{t:main}, under the hypothesis that $\partial \left[ T \right] = 0$. They were motivated by the will to prove the analogue of Theorem \ref{p_mass_chain:thm} above when the ambient space is a compact convex subset of a Banach space with mild additional assumptions. Even though our theorem covers also the case with boundary, our proof is considerably simpler than theirs, essentially because we can rely on the polyhedral approximation theorem, which is not available in their context. Actually, our result would follow directly from theirs if one could independently guarantee the validity of the following proposition. However, we were not able to devise a proof independent of Theorem \ref{t:main}. 

\begin{proposition} \label{0sum_prop}
Let $\left[ T \right] \in \I_{1,K}^{p}(\R^n)$. Then, for any $R = \sum_{i=1}^{q} \theta_{i} \delta_{x_{i}} \in \Rc_{0,K}(\R^n)$ such that $R = \partial T \, \modp$ one has:
\begin{equation} \label{0sum}
\sum_{i=1}^{q} \theta_{i} \equiv 0 \, ({\rm mod}\, p).
\end{equation}
\end{proposition}
 
Assume the validity of the Proposition. An alternative proof of our Theorem \ref{t:main} can be obtained as follows. Let $\left[ T \right] \in \I_{1,K}^{p}(\R^m)$ and let $R = \sum_{i=1}^{q} \theta_{i} \delta_{x_{i}} \in \Rc_{0,K}(\R^n)$ be such that $R = \partial T \, \modp$. Fix $x_{0} \notin \{x_{1}, \dots, x_{q}\}$ and consider the cone $C$ with vertex $x_{0}$ over $R$, i.e. the integral $1$-current 
\begin{equation}
C := \sum_{i=1}^{q} \llbracket S_{i}, \tau_{i}, \theta_{i} \rrbracket,
\end{equation}
where $S_{i}$ is the segment joining $x_{i}$ to $x_{0}$ and $\tau_{i}:= \displaystyle \frac{x_{0} - x_{i}}{|x_{0} - x_{i}|}$. By \eqref{0sum}, the multiplicity of $\partial C$ at $x_{0}$ is an integer multiple of $p$, and thus via a simple computation $\partial (T + C) = 0 \, \modp$. Applying the result of Ambrosio and Wenger, we finally obtain that there exists an integral current $J \in \I_{1,K}(\R^n)$ such that $J = T + C \, \modp$. Hence, $I := J - C$ is an integral current with $I = T \, \modp$. 
\\

Although the analogue of Proposition \ref{0sum_prop} for classical currents is a well known fact (i.e. the sum of the multiplicities in the boundary of an integral $1$-current is zero), the validity of Proposition \ref{0sum_prop} does not follow trivially. Nevertheless, it can in fact be deduced as a consequence of our Corollary \ref{dim0_cor}.

\begin{proof}[Proof of Proposition \ref{0sum_prop}]
Let $T$ and $R$ be as in the statement. Then, since $\Fl_{K}^{p}(\partial T - R) = 0$, Corollary \ref{dim0_cor} implies the existence of currents $Q \in \Rc_{0,K}(\R^n)$ and $S \in \Rc_{1,K}(\R^n)$ such that
\[
\partial T - R = p(Q + \partial S),
\]
that is
\[
\partial(T - pS) = R + pQ.
\]
In particular, $T - pS$ is a classical integral current, and thus the sum of the multiplicities in $R$ must equal that of $-pQ$, which concludes the proof.
\end{proof}

\nocite{*}
\bibliographystyle{plain}
\bibliography{References}

\Addresses

\end{document}